\documentclass{amsart}

\DeclareMathOperator{\spl}{sp}

\newcommand\blfootnote[1]{%
  \begingroup
  \renewcommand\thefootnote{}\footnote{#1}%
  \addtocounter{footnote}{-1}%
  \endgroup
}

\usepackage[T1]{fontenc}
\usepackage{tikz,xcolor}
\usetikzlibrary{matrix}

\usepackage{faktor,amssymb}
\usepackage{tikz-cd}
\usepackage{hyperref}
\hypersetup{
    colorlinks,
    citecolor=black,
    filecolor=black,
    linkcolor=black,
    urlcolor=black
}
\usepackage{braket}
\usetikzlibrary{calc,
                matrix
                }
\usetikzlibrary{arrows}
\usepackage{csquotes}
\newtheorem{theorem}[subsection]{Theorem}

\newtheorem{proposition}[subsection]{Proposition}

\newtheorem{corollary}[subsection]{Corollary}
\newtheorem{example}[subsection]{Example}

\title[Consequences of the $\mu$-constant condition]{Some consequences of the $\mu$-constant condition for families of surfaces}
\author{MARTA ALDASORO ROSALES}
\date{}

\begin{document}
\maketitle

\begin{abstract}
Let $f : X\to \Delta$ be a $1$-parameter family of $2$-dimensional isolated hypersurface singularities. In this paper, we show that if the Milnor number is constant, then any semistable model, obtained from $f$ after a sufficiently large base change must satisfy non trivial restrictions. Those restrictions are in terms of the dual complex, Hodge structure and numerical invariants of the central fibre.
\end{abstract}


\section{Introduction}

\blfootnote{2020 \textit{Mathematics Subject Classification}. \subjclass{14B05, 14B07, 32B10}}
Let $f : (\mathbb{C}^n,0) \to (\mathbb{C},0)$ be a holomorphic function germ, the \textit{Milnor number} of $f$ is $\mu(f) := \dim_\mathbb{C} \mathbb{C}[[z_1,...,z_n]] / \langle \frac{\partial f}{\partial z_1}, \cdots \frac{\partial f}{\partial z_n}\rangle$. Let $\Delta$ be a disk and consider $f_t : (\mathbb{C}^3,0) \to (\mathbb{C},0)$ to be a family of isolated surface singularities smoothly depending holomorphically on the parameter $t \in \Delta$, with constant Milnor number $\mu$ at the origin. In the second half of the 1970s, a result of L\^e and Ramanujam \cite{11}, combined with work by Timourian \cite{17} and King \cite{6}, makes it possible to prove that if $n\neq 3$ and if $\mu$ is constant then the family is topologically trivial. A long-standing problem is to prove the same statement for $n=3$. For this, it is necessary that, after a finite base change, the total space $X\to (\mathbb{C},0)$ of the family admits a very weak simultaneous resolution as defined by Laufer (see Thm. 6.4 in \cite{10}). Therefore, as pointed out in \cite{2}, a possible way of splitting the L\^e-Ramanujam problem in two pieces is to prove first that the $\mu$ constant condition implies very weak simultaneous resolution after a finite base change, and then try to prove topological triviality using the very weak simultaneous resolution. 

Note that if we define $F(x,t):=f_t(x)$ the total space $X\subset \mathbb{C}^3\times\Delta$ is defined by $\{F=0\}$. Recall that a very weak simultaneous resolution is a flat proper birational morphism $Y\to X$ from a smooth threefold such that for all $t \in \Delta$, $Y_t\to X_t$ is a resolution of singularities. Using the Semistable Reduction theorem, see \cite{7}, Thm. 7.17, we can assume that, after a base change, the total space $X\to (\mathbb{C},0)$ of the family admits a semistable resolution $\pi:W\to X$. Since the problem is local in $\Delta$ we may safely assume that all the exceptional divisors contained in fibres over points of $\Delta$ are contained in the fibre $W_0$. If a very weak simultaneous resolution exists, the semistable resolution may be chosen equal to it, \color{black} and therefore there are no exceptional divisors contained in $W_0$.

Given any resolution $\pi:W\to X$, let $Z\subset W$ be the exceptional divisor, we have $W \setminus W_0$ is a topologically trivial family of surfaces over $\Delta\setminus\{0\}$. Then, by a result of Laufer, see \cite{10}, Lemma 6.2, we can assume that $Z\setminus Z_0$ forms a locally trivial family of curves over $\Delta\setminus\{0\}$, and then the monodromy is a permutation and hence of finite order. So we may assume that our semistable resolution induces the trivial permutation, matching the situation of the simultaneous resolution of a topologically trivial family.

This motivates the results of our paper: we will find restrictions on the possible configurations of exceptional divisors contained in $W_0$. The restrictions will be of combinatorial, Hodge theoretic and topological nature. We would like to see that the total transform of the central fibre has invariants as similar as possible to those of the strict transform. However, as we work in a semistable resolution, our aim can not be to prove that there are no more components in the total transform apart from the strict transform, since that is not true for general semistable resolution, as we can see in the following example.

\begin{example}\label{example}
    Consider the family $ \ \mathbb{C}^2 \times \Delta \to \Delta$, which has constant Milnor number at the origin and perform a blow up with center a rational curve over the central fibre. We get a semistable resolution of the family, where the central fibre has two irreducible components.
\end{example}

\begin{theorem}
Let $\pi : W \to X$ be any semistable resolution obtained after a finite base change, such that the permutation that the monodromy induces on the components of the exceptional divisor $Z_t$ is trivial. Let $\overline{\pi} : \mathfrak{W} \to \mathfrak{X}$ and $\overline{f} : \mathfrak{X} \to \Delta$ be projective morphisms such that $X \subset \mathfrak{X}$, $W \subset \mathfrak{W}$ are open embeddings, $\overline{\pi}_{|W} = \pi$, $\overline{f}_{|X} = f$, $\overline{f}$ have no critical points outside $X$ and $\overline{\pi}_{|\mathfrak{W} \setminus W} : \mathfrak{W} \setminus W \to \mathfrak{X} \setminus X$ is an isomorphism. Then,
\begin{enumerate}
 \item the dual complex (for definition, see \cite{1}) associated to the central fibre of $W$ has the rational homology of a point.
 \item the second cohomology groups of the exceptional divisors of the semistable resolution lying in the central fibre have $h^{2,0}=0$, that is, they have pure Hodge structures of type $(1,1)$ in the middle cohomology group.
 \item the first Betti number of the central fibre of $\mathfrak{W}$ coincides with the first Betti number of the strict transform of $\mathfrak{X}_0$.
 \item the sum of the triple intersections of the irreducible components of the central fibre of $\mathfrak{W}$ equals $-3$ times the number of triple intersections.
\end{enumerate}
\end{theorem}

The way to prove the first two results of the theorem is first use the $\mu$-constant condition to construct a trivial monodromy action on the cohomology groups of a generic fibre of $\mathfrak{W}$ with coefficients in $\mathbb{Q}$. This will give constraints in the first page of the Steenbrink spectral sequence, degenerating to the cohomology groups of these fibres endowed with their limit mixed Hodge structure. These constraints together with some tools from both algebraic topology and mixed Hodge theory will give us the results. To show the third part, we use part (1) and (2) of the theorem together with the fact that a flat projective morphism has constant analytic Euler characteristic over the fibres and that a family of isolated singularities with constant Milnor number has constant geometric genus. For the fourth part, we combine the three previous results with a condition obtained from Steenbrink's spectral sequence and with the technique used by Kollar and Shepherd Barron in their work \cite{8}.

I would like to thank to Javier Fernández de Bobadilla for proposing me the problem and guiding me through it. I am also grateful to Tomasz Pełka for helpful discussions and ideas he has given me in the process.

\subsection*{Acknowledgements}
This research was supported by the grants: Programa Predoctoral de Formación de
Personal Investigador No Doctor of the Basque Government Department of Education, the projects PID2020-117080RB-C55 and PID2020-114750GB-C33 from Spanish Ministry of Science, the projects SEV-2017-0718 and SEV-2023-2026 from Severo Ochoa and FEDER "Una manera de hacer Europa". It was also supported by the Basque Government through the BERC 2022-2025 program and by the Ministry of Science and Innovation: BCAM Severo Ochoa accreditation CEX2021-001142-S / MICIN / AEI / 10.13039/501100011033.

\section{Preliminaries}

\subsection{A homologically trivial cobordism}

For background on the Milnor fibration, see \cite{13}.

A family of hypersurface singularities is given by the germ at $\{O\}\times \Delta_\xi$ of a holomorphic map $F : \mathbb{C}^3 \times \Delta_\xi \to \mathbb{D}_\delta \times \Delta_\xi$ of the form $F(z_1,z_2,z_3,t) = (f_t(z_1,z_2,z_3),t)$. Assume that $f_t$ has an isolated singularity at the origin with Milnor number independent of $t$. Fix $t \in \Delta_{\xi}^{\ast}$ and let $\epsilon_t < \epsilon_0$ and $\delta_t < \delta_0$, Milnor radii for both $f_t$ and $f_0$, respectively. Then, by \cite{11}, after shrinking $\xi > 0$, the restriction
\begin{align*}
    F : (\mathbb{B}_{\epsilon_0} \times \Delta_{\xi}) \cap F^{-1}(\mathbb{D}_{\delta_0}^* \times \Delta_{\xi}) \to \mathbb{D}_{\delta_0}^* \times \Delta_{\xi}
\end{align*}
is a locally trivial fibration and the restriction
\begin{align}\label{fib}
    f_t : \mathbb{B}_{\epsilon_0} \cap f_t^{-1}(\mathbb{D}_{\delta_t}^{\ast}) \to \mathbb{D}^{\ast}_{\delta_t}
\end{align}
is diffeomorphic to the Milnor fibration in the tube for $f_0$. Futhermore, the Milnor fibration in the tube for $f_t$ is contained in the fibration (\ref{fib}). For any $z \in \mathbb{D}_{\delta_t}$, define 
\begin{align*}
    f^{-1}_t(z) \cap (\mathbb{B}_{\epsilon_0} \setminus \overset{\circ}{\mathbb{B}}_{\epsilon_t}), 
\end{align*}
to be the difference between the fibre of (\ref{fib}), and the Milnor fibre of $f_t$ over $z$. It is easy and proven in \cite{11} that this difference is diffeomorphic to 
\begin{align*}
    f^{-1}_t(0) \cap (\mathbb{B}_{\epsilon_0} \setminus \overset{\circ}{\mathbb{B}}_{\epsilon_t}).
\end{align*}
It is a manifold with two boundary components, each of them diffeomorphic to the links of $f_t$ and $f_0$, and we regard it as a cobordism. Under the  $\mu$-constant condition it is proved in \cite{11} that  this cobordism is homologically trivial. That is, we have 
\begin{equation}\label{trivcobor}
H_*(  f^{-1}_t(0)\cap (\mathbb{B}_{\epsilon_0} \setminus \overset{\circ}{\mathbb{B}}_{\epsilon_t}),f^{-1}_t(0)\cap \partial\mathbb{B}_{\epsilon_t}   ;\mathbb{Z})=0 
\end{equation}

\subsection{Geometric set-up and notation} \label{set}
 Consider $f_t : (\mathbb{C}^3,0) \to (\mathbb{C},0)$ be a $\mu$-constant family so that the restriction
\begin{align*}
    F : (\mathbb{B}_{\epsilon_0} \times \Delta) \cap F^{-1}(\mathbb{D}_{\delta_0}^* \times \Delta) \to \mathbb{D}_{\delta_0}^* \times \Delta
\end{align*}
is a locally trivial fibration, as we saw in previous section. Since the singularity is isolated, by finite determinacy (see \cite{12} and \cite{18}), we can assume that $f_t$ are polynomials. 

Let $X_t := f^{-1}_t(0) \cap \mathbb{B}_{\epsilon_0}$, where $\epsilon_0$ is the Milnor radius for the ball of $f_0$ and consider the family $f := F_{|F^{-1}(\{0\} \times \Delta)} : X \to \Delta$, where $X = \cup_{t \in \Delta} \{(X_t,t)\}$. Compactify $\mathbb{C}^3$ to $\mathbb{P}^3$, and let $\mathfrak{X}_t$ be the closure in $\mathbb{P}^3$ of $f^{-1}_t(0)$. After adding to $f_t$ a sufficiently general high order homogeneous term (the same for all $t$), we may assume that $\mathfrak{X}_t$ has the origin as its only singularity. Denote $\mathfrak{X} := \cup_{t \in \Delta} \{(\mathfrak{X}_t,t)\}$ and denote $\bar{f} : \mathfrak{X} \to \Delta$ the projective morphism sending $\mathfrak{X}_t$ to $t$. 

Let $\pi : \mathfrak{W} \to \mathfrak{X}$ to be a semistable resolution of $\bar{f} : \mathfrak{X} \to \Delta$, where if needed we make first a base change according to Semistable Reduction Theorem mentioned before. Let $W := \pi^{-1}(X)$ and $W_t := \pi^{-1}(X_t)$.
For $t\neq 0$, we consider the restriction $\pi_t : W_t \to X_t$, and let $Z_t$ denote the exceptional divisor of $\pi_t$. Let $\mathfrak{W}_0 = E_0 \cup \cup_{i=1}^k E_i$, where $E_0$ is the strict transform of $\mathfrak{X}_0$; then we have $W_0=\tilde{W}_0 \cup \cup_{i=1}^k E_i$, where $\tilde{W}_0 = E_0 \cap W$. Denote $V_t := \mathfrak{W}_t \setminus W_t$. 

Let $Z := \overline{\cup_{t \neq 0} Z_t}$, where the closure is taken in $\mathfrak{W}$ and denote $\mathfrak{W}^{\ast} := \mathfrak{W} \setminus \mathfrak{W}_0$, $W^{\ast} := W \setminus W_0$, $Z^{\ast} := Z \cap \mathfrak{W}^{\ast}$ and $\Delta^{\ast} := \Delta \setminus \{0\}$. Note that ($\mathfrak{W}^{\ast}, W^{\ast}, Z^{\ast}) \to \Delta^{\ast}$ is a locally trivial fibration, where the restriction to $\mathfrak{W}^{\ast} \setminus W^{\ast} \to \Delta^{\ast}$ is trivial. Indeed, by construction, $\mathfrak{W} \setminus W \to \Delta$ is locally trivial and since $\Delta$ is contractible, we have $\mathfrak{W} \setminus W \to \Delta$ is trivial.

\section{Construction of the homologically trivial monodromy}


        

    

On the other hand, using that $\pi$ is an isomorphism out of the singular locus, excision and that $Z_t$ is a deformation retract of $\pi^{-1}(X_t \cap \mathbb{B}_{\epsilon_t})$, we get, with the notations of the previous section,     
    \begin{align*}
        H_i(X_t \cap (\mathbb{B}_{\epsilon_0} \setminus \overset{\circ}{\mathbb{B}}_{\epsilon_t}), X_t \cap \mathbb{S}_{\epsilon_t}) = H_i(\pi^{-1}(X_t \cap (\mathbb{B}_{\epsilon_0} \setminus \overset{\circ}{\mathbb{B}}_{\epsilon_t})), \pi^{-1}(X_t \cap \mathbb{S}_{\epsilon_t})) = \\
       = H_i(\pi^{-1}(X_t \cap \mathbb{B}_{\epsilon_0}), \pi^{-1}(X_t \cap \mathbb{B}_{\epsilon_t})) = H_i(W_t,Z_t).
    \end{align*}

So, combining with the vanishing~(\ref{trivcobor}), we obtain the equivalent statements
\begin{equation}
\label{rrr}
H_i(W_t,Z_t,\mathbb{Q})=0,\quad\quad\quad H_i(Z_t;\mathbb{Q})\stackrel{\cong}{\longrightarrow} H_i(W_t;\mathbb{Q}).
\end{equation}

\color{black}

\begin{proposition}\label{a}
    In the setting of Subsection \ref{set}, the monodromy action on $H^i(\mathfrak{W}_t, \mathbb{Q})$ is trivial, for all $i \in \mathbb{N}$.
\end{proposition}

\begin{proof}
    As we noticed in Subsection \ref{set}, ($\mathfrak{W}^{\ast},W^{\ast},Z^{\ast}) \to \Delta^{\ast}$ is a locally trivial fibration, where  $\mathfrak{W}^{\ast}\setminus W^{\ast} \to \Delta^{\ast}$ is trivial. Thus, there exists a geometric monodromy $h : (\mathfrak{W}_t, W_t, Z_t) \to (\mathfrak{W}_t, W_t, Z_t)$ such that $h_{|\mathfrak{W}_t \setminus W_t} = \mathrm{Id}$. Moreover, we have chosen our base change such that each irreducible component of $Z_t$ is sent to itself by the monodromy. 
    
    Now, on $H_0(W_t) \cong \mathbb{Q}, H_3(W_t) = 0$ and $H_4(W_t) = 0$ the monodromy action is obviously trivial, and since $H_2(W_t) \cong H_2(Z_t)$, and the monodromy induces the trivial permutation at the set of irreducible components of $Z_t$, is also trivial on $H_2(W_t)$. Finally, on one hand, we claim that $H_1(W_t) \cong H_1(\partial W_t)$, which will be shown in (2) below. On the other hand, the geometric monodromy on $\partial W_t$ being the identity implies that the monodromy action on $H_i(\partial W_t)$ is trivial. Therefore, we get that the monodromy action is trivial on each $H_i(W_t)$. 

    Now, to prove that it is also trivial on $H_i(\mathfrak{W}_t, \mathbb{Q})$, we consider the Mayer Vietoris long exact sequence (see \cite{5} Ch. 2):
    \begin{multline*}
    0 \to H_4(\mathfrak{W}_t) \cong \mathbb{Q} \to H_3(\partial W_t) \cong \mathbb{Q} \to 0 \oplus H_3(V_t) \to H_3(\mathfrak{W}_t) \to H_2(\partial W_t) \to \\
    \to H_2(W_t) \oplus H_2(V_t) \to H_2(\mathfrak{W}_t) \to H_1(\partial W_t) \to H_1(W_t) \oplus H_1(V_t) \to H_1(\mathfrak{W}_t) \to 0
    \end{multline*}
    We will denote by $\mathbb{T}_i^{(-)}$ the monodromy action on $H_i(-, \mathbb{Q})$.
    \begin{enumerate}
        \item [(1)] We have that $H_1(W_t) \oplus H_1(V_t) \to H_1(\mathfrak{W}_t)$ is surjective and, since assuming the claim above, $\mathbb{T}_1^{W_t} = \mathbb{T}_1^{V_t} = \mathrm{Id}$ and the monodromy action commutes with the morphism, we get $\mathbb{T}_1^{\mathfrak{W}_t} = \mathrm{Id}$.

        \item [(2)] If we had $H_2(W_t) \oplus H_2(V_t) \to H_2(\mathfrak{W}_t)$ to be surjective by the same argument we will get that $\mathbb{T}_2^{\mathfrak{W}_t} = \mathrm{Id}$. For this, we show that $H_1(\partial W_t) \to H_1(W_t) \oplus H_1(V_t)$ is injective, and for this, the claim that $H_1(\partial W_t) \to H_1(W_t)$ is an isomorphism suffices. Now, to prove it, consider 
        $$
        H_2(W_t) \xrightarrow{\alpha} H_2(W_t, \partial W_t) \to H_1(\partial W_t) \to H_1(W_t) \to H_1(W_t, \partial W_t)
        $$
        We have the following commutative diagram where the vertical arrows are isomorphisms:
        \begin{center}
        \begin{tikzcd}
        H_2(W_t) \arrow[r, "\alpha"]  & H_2(W_t, \partial W_t)  \arrow[d, "\cong"] \\
         & H_2(W_t, W_t \cap \pi^{-1}(\mathbb{B}_{\epsilon_0} \setminus \overset{\circ}{\mathbb{B}}_{\epsilon_t})) \\
         H_2(W_t \cap \pi^{-1}(\mathbb{B}_{\epsilon_t})) \arrow[uu, "\cong"] \arrow[r] & H_2(W_t \cap \pi^{-1}(\mathbb{B}_{\epsilon_t}), \partial (W_t \cap \pi^{-1}(\mathbb{B}_{\epsilon_t}))) \arrow[u, "\cong"]
        \end{tikzcd}
        \end{center}
        Indeed, the isomorphism on the left is obtained from the Mayer-Vietoris long exact sequence for the decomposition 
        \begin{align*}
            W_t=W_t \cap \pi^{-1}(\mathbb{B}_{\epsilon_0} \setminus \overset{\circ}{\mathbb{B}}_{\epsilon_t})\cup W_t \cap \pi^{-1}(\mathbb{B}_{\epsilon_t})
        \end{align*}
        and the vanishing of Equation~(\ref{trivcobor}).  
        
        The bottom right isomorphism comes from excision and the top one from the long exact sequence of the triple $(W_t, W_t \cap \pi^{-1}(\mathbb{B}_{\epsilon_0} \setminus \overset{\circ}{\mathbb{B}}_{\epsilon_t}), \partial W_t)$, see \cite{5}, pg. 118:
        \begin{center}
        $H_i(W_t \cap \pi^{-1}(\mathbb{B}_{\epsilon_0} \setminus \overset{\circ}{\mathbb{B}}_{\epsilon_t}), \partial W_t) \to H_i(W_t, \partial W_t) \to H_i(W_t, W_t \cap \pi^{-1}(\mathbb{B}_{\epsilon_0} \setminus \overset{\circ}{\mathbb{B}}_{\epsilon_t})),$
        \end{center}
        and the vanishing of Equation~(\ref{trivcobor}).
        
        Now, Lefschetz duality gives us $H_2(W_t \cap \pi^{-1}(\mathbb{B}_{\epsilon_t}), \partial (W_t \cap \pi^{-1}(\mathbb{B}_{\epsilon_t}))) \cong H^2(W_t \cap \pi^{-1}(\mathbb{B}_{\epsilon_t}))$ and the morphism $H_2(W_t \cap \pi^{-1}(\mathbb{B}_{\epsilon_t})) \to H^2(W_t \cap \pi^{-1}(\mathbb{B}_{\epsilon_t}))$ is an isomorphism since $W_t \cap \pi^{-1}(\mathbb{B}_{\epsilon_t})$ retracts to $Z_t$ and the intersection form is nondegenate at the resolution. We conclude that $\alpha$ is an isomorphism. Besides, using Lefschetz duality $H_1(W_t,\partial W_t) \cong H^3(W_t) = 0$, we get $H_1(\partial W_t) \xrightarrow{\cong} H_1(W_t)$ as needed.

        \item[(3)] The intersection form $H_3(\mathfrak{W}_t) \otimes H_1(\mathfrak{W}_t) \to \mathbb{Q}$ is a perfect pairing which is preserved by the monodromy. We have 
        \begin{align*}
            <[z],[w]> = <\mathbb{T}_3([z]),\mathbb{T}_1([w])> = <\mathbb{T}_3([z]),[w]>,
        \end{align*}
        then $<\mathbb{T}_3([z]) - [z],[w]> = 0$ for all $[w] \in H_1(\mathfrak{W}_t)$. So since the map is a perfect pairing, we conclude $\mathbb{T}_3([z]) = [z]$.
    \end{enumerate}
    
    Now, using the Poincar\'e Duality isomorphisms $H^i(\mathfrak{W}_t) \cong H_{n-i}(\mathfrak{W}_t)$, we get the desired result.
\end{proof}

\section{The Steenbrink spectral sequence}

In this section, our goal is to recall the Steenbrink spectral sequence and obtain restrictions on its first page. For all the results and definitions we refer to \cite{15}.

We start with the proper regular function $g := \bar{f} \circ \pi : \mathfrak{W} \to \Delta$ defined in Subsection \ref{set}. For simplicity, we will denote $E = \mathfrak{W}_0$ and we set the following notation:
\begin{flushleft}
    \hspace{12.5mm} $E_I = \cap_{i \in I} E_{i}$, where $I \subset \{0,...,k\}$; \\
    \hspace{12.5mm} $E(m) = \bigsqcup _{|I|=m} E_I$, where $m \in \{1,...,k+1\}$; \\
    \hspace{12.5mm} $a_I : E_I \hookrightarrow \mathfrak{W};$ \\
    \hspace{12.5mm} $a_m = \bigsqcup_{|I| = m} a_I : E(m) \to \mathfrak{W}.$
\end{flushleft}

Introduce the universal cover $e : \mathfrak{h} \to \Delta^{\ast}$, where $e(u) = \exp(2 \pi iu)$ and put
\begin{equation*}
    \mathfrak{W}_\infty := \{(w,u) \in \mathfrak{W} \times \mathfrak{h} \ | \ f(w) = e(u) \}.  
\end{equation*}
Let $k : \mathfrak{W}_\infty \to \mathfrak{W}$ be the projection on $\mathfrak{W}$ and $g_\infty := g \circ k : \mathfrak{W}_\infty \to \Delta$. Note that the fibre $\mathfrak{W}_\infty$ is homotopic to any fibre $\mathfrak{W}_t$ of $g$, whenever $t \neq 0$, since $g_\infty$ is differentiably a product.

Moreover, the total space $\mathfrak{W}$ is homotopy equivalent to $E$ by a strong deformation retraction $r : \mathfrak{W} \to E$. So the composition of the inclusion $i_t : \mathfrak{W}_t \hookrightarrow \mathfrak{W}$ followed by the retraction can be seen as the \textit{specialization map} $r_t : \mathfrak{W}_t \to E$. 

Then, Cor. 11.23 in \cite{15} tells us that the spectral sequence 
\begin{equation}\label{14}
    E_1^{-r,q+r} = \oplus_{k \geq 0,-r} H^{q-r-2k}(E(r+2k+1))(-r-k) \Longrightarrow H^q(\mathfrak{W}_\infty,\mathbb{Q})
\end{equation}
degenerates at the term $E_2$, where $E_1^{-r,q+r}$ is a Hodge structure of weight $q + r$. Moreover, it endows $H^q(\mathfrak{W}_\infty,\mathbb{Q})$ with a mixed Hodge structure.

\begin{proposition}\label{rr}
    The first page (\ref{14}) of the Steenbrink spectral sequence is the one shown below, it satisfies the surjectivity and injectivity conditions shown on it and is exact at $E^{-1,4}$ and at $E^{1,0}$. 
\end{proposition}

\tikzset{ 
table/.style={
  matrix of nodes,
  row sep=-\pgflinewidth,
  column sep=-\pgflinewidth,
  nodes={rectangle,align=center},
  text depth=1.25ex,
  text height=2.5ex,
  nodes in empty cells
},
row 6/.style={nodes={fill=green!10,text depth=0.4ex,text height=2ex}},
column 1/.style={nodes={text width=0.6em, fill=green!10}},
column 2/.style={nodes={text width=6.7em}},
column 3/.style={nodes={text width=6.7em}},
column 4/.style={nodes={text width=11.6em}},
column 5/.style={nodes={text width=4.8em}},
column 6/.style={nodes={text width=4.8em}},
}
\hspace*{-1cm}\begin{tikzpicture}

  \matrix (m) [table]
  {
          4     &  $H^0(E(3))(-2)$ &  $H^2(E(2))(-1)$  & $H^4(E(1)$) & 0 & 0 \\
          3     &  0 &  $H^1(E(2))(-1)$ & $H^3(E(1))$  & 0 & 0 \\   
          2     &  0  &  $H^0(E(2))(-1)$   &  $H^2(E(1)) \oplus H^0(E(3))(-1)$ &  $H^2(E(2))$  &  0  \\
          1     &  0 &  0  & $H^1(E(1))$  & $H^1(E(2))$ & 0 \\
          0     &  0  & 0 &  $H^0(E(1))$  & $H^0(E(2))$  & $H^0(E(3))$ \\
    \quad\strut & -2 & -1 &  0  &  1  &  2  \strut \\};

    \foreach \x in {1,...,5}
{
  \draw 
    ([xshift=-.5\pgflinewidth]m-\x-1.south west) --   
    ([xshift=-.5\pgflinewidth]m-\x-6.south east);
  }
\foreach \x in {1,...,5}
{
  \draw 
    ([yshift=.5\pgflinewidth]m-1-\x.north east) -- 
    ([yshift=.5\pgflinewidth]m-6-\x.south east);
}  

\draw[thick] (m-1-1.north west) -- (m-6-1.south west);
\draw[thick] (m-1-6.north east) -- (m-6-6.south east);
\draw[thick] (m-1-1.north west) -- (m-1-6.north east);
\draw[thick] (m-6-1.south west) -- (m-6-6.south east);
    
\begin{scope}
\draw[shorten <=11.25mm, shorten >=11.25mm,right hook->] (m-1-2.center) -- (m-1-3.center) ;
\draw[shorten <=11.5mm, shorten >=11.25mm,->] (m-1-3.center) -- (m-1-4.center);
\draw[shorten <=11.5mm, shorten >=11.25mm,right hook->] (m-2-3.center) -- (m-2-4.center);
\draw[shorten <=11.35mm, shorten >=20.8mm,right hook->] (m-3-3.center) -- (m-3-4.center);
\draw[shorten <=20.8mm, shorten >=8mm,->>] (m-3-4.center) -- (m-3-5.center);
\draw[shorten <=11.25mm, shorten >=8mm,->>] (m-4-4.center) -- (m-4-5.center);
\draw[shorten <=12.25mm, shorten >=8mm, ->] (m-5-4.center) -- (m-5-5.center);
\draw[shorten <=7mm, shorten >=8mm,->>] (m-5-5.center) -- (m-5-6.center);
\end{scope}

\end{tikzpicture}

\begin{proof}
    The spectral sequence degenerates at $E_2$. Since the weight filtration on the limit mixed Hodge structure is the monodromy weight filtration, see Cor. 11.42 in \cite{15}, and the monodromy action on $H^\ast(\mathfrak{W}_t)$ is trivial, we get that the second page is concentrated at the central column, giving the desired result.
\end{proof}

\begin{corollary}
    The limit mixed Hodge structure that the constructed complex puts on $H^k(\mathfrak{W}_\infty)$ is in fact a pure Hodge structure of weight $k$. 
\end{corollary}

\section{Main results}

We will finish this paper giving the proofs for the main theorems. 

\begin{theorem}\label{1imp}
    The rational homology of the dual complex associated to the central fibre $\tilde{W}_0$ is isomorphic to rational homology of a point.
\end{theorem}

\begin{proof}
    Since the spectral sequence degenerates at $E_2$ and $H^0(\mathfrak{W}_\infty,\mathbb{Q}) \cong \mathbb{Q}$, we get that the kernel of the map $H^0(E(1),\mathbb{Q}) \to H^0(E(2),\mathbb{Q})$ on the first page of the spectral sequence of Prop. \ref{rr} is isomorphic to $\mathbb{Q}$. Thus, using exactness of the bottom row, and noting that the bottom row compute the rational cohomology of the dual complex of $E$ and that the dual complex of $E$ and $\tilde{W}_0$ are the same, we conclude.
    
\end{proof}

\begin{theorem}\label{2imp}
    The second cohomology groups of the exceptional divisors $E_i$ for $i > 0$, are Hodge structures of pure type $(1,1)$.
\end{theorem}

\begin{proof}
    Denote $i : E_i \cap (\cup_{j \neq i} E_j) \hookrightarrow E_i$ to the inclusion. Then, it follows from \cite{15} Prop. 5.54, that we have the following exact sequence of mixed Hodge structures
    \begin{align*}
        H_c^{2}(\mathring{E}_i) \to H^2(E_i) \to H^2(E_i \cap (\cup_{j \neq i} E_j))
    \end{align*}
    Now, assume there exists a non-zero cohomology class $[\alpha] \in H^{2,0}(E_i)$. Then, using the exact sequence and that $H^2(E_i \cap (\cup_{j \neq i} E_j))$ is a pure Hodge structure of type $(1,1)$ (see \cite{15}, Prop. 1.11), there is a class $[\gamma] \in H_c^{2}(\mathring{E}_i)$, being sent to $[\alpha]$. So, we can assume that the support of $\alpha$ is contained in $\mathring{E}_i$.

    Following the discussion prior to the Thm. 11.29 of \cite{15}, we get a lifting of the specialization map to the level of mixed Hodge complexes of sheaves that makes the specialization map 
    \begin{align}\label{lk}
        \spl^{\ast} : H^{\ast}(E) \to H^{\ast}(\mathfrak{W}_\infty)
    \end{align}
    a morphism of mixed Hodge structures. Moreover, there is a spectral sequence degenerating to the cohomology of $H^{\ast}(E)$ and a homomorphism from this spectral sequence to the Steenbrink spectral sequence that at the limit gives the morphism (\ref{lk}) in the category of mixed Hodge structures. The first page of the spectral sequence converging to $H^{\ast}(E)$ is:

    \begin{align*}
        E_1^{p,q} = H^{q}(E(p+q+1), \mathbb{Q}) 
    \end{align*}

    Since $H^2(E(2),\mathbb{Q})$ is of type $(1,1)$, the image of $[\alpha]$ in $H^2(E(2),\mathbb{Q})$ is zero, thus, $[\alpha]$ induces a class $[\beta] \in H^{2,0}(E,\mathbb{Q})$ which is non-zero and whose support is in $\mathring{E}_i$. 
    
    Thus, through the specialization morphism we get a class $[\spl^\ast \beta] \in H^2(\mathfrak{W}_\infty)$ which will not vanish. Indeed, $[\alpha] \in H^{2,0}(E_i)$, and $H^0(E(2)), H^2(E(2))$ are pure Hodge structures of type $(1,1)$. So, the image of $[\alpha]$ in $H^2(E(2),\mathbb{Q})$ is zero and $[\alpha]$ is not the image of a cycle of $H^0(E(2))$. Then using Steenbrink spectral sequence we get that $[\spl^\ast \beta]$ is non-zero.

    Now, recall the retraction morphism $r : \mathfrak{W} \to E$ and the projection morphism $k : \mathfrak{W}_\infty \to \mathfrak{W}$. Since $\mathring{E}_i$ is an open subset of $E$, the support of $r^\ast \beta$ must be contained in a neighborhood of a compact set $K \subset \mathring{E}_i$ in $\mathfrak{W}$. Furthermore, since $\mathring{E}_i \subset E \setminus E_0$, the neighbourhood is contained in $W$. Then, the support of $\spl^\ast \beta = k^\ast r^\ast \beta$ is contained in $k^{-1}W = W_\infty$.

    Consider the following injective morphism given in the proof of Prop. \ref{a},
    \begin{align*}
        0 \to H^2(\mathfrak{W}_\infty) \hookrightarrow H^2(W_\infty) \oplus H^2(V_\infty).
    \end{align*}
    The support of $\spl^\ast \beta$ is contained in $W_\infty$, thus, $[\spl^\ast \beta] \neq [0]$ in $H^2(W_\infty)$. 
    
    In order to complete the proof, recall that we denote by $Z$ the closure in $\mathfrak{W}$ of $\cup_{t\neq 0} Z_t$ and let $\pi : \tilde{Z} \to Z$ denote a semistable resolution of $Z \to \Delta$, where if nedeed we first make a base change. Then, we have a morphism $\tilde{Z} \to \mathfrak{W}$ and using \cite{3}, Thm. 7.4., we get a morphism of mixed Hodge structure $H^2(\mathfrak{W}_\infty) \to H^2(\tilde{Z}_\infty)$. Now, since the $H^2(W_\infty) \to H^2(\tilde{Z}_\infty)$ is an isomorphism, $[\spl^\ast \beta] \neq [0]$ in $H^2(\tilde{Z}_\infty)$. Then, if we prove that $H^2(\tilde{Z}_\infty)$ is of type $(1,1)$, we will get a contradiction. But this follows from \cite{15}, Cor. 11.25 and the fact that $H^2(\tilde{Z}_t)$ is a pure Hodge structure of type $(1,1)$.
\end{proof}

\begin{theorem}\label{3imp}
    The first Betti number of $E$ is the same as the first Betti number of $E_0$.
\end{theorem}

\begin{proof}
    Consider the following short exact sequence (see \cite{15}, Example 7.23):

\begin{align*}
    0 \to \mathcal{O}_{E} \to (a_0)_{\ast}\mathcal{O}_{E(1)} \to (a_1)_{\ast}\mathcal{O}_{E(2)} \to (a_2)_{\ast}\mathcal{O}_{E(3)} \to 0.
\end{align*}

    

Using that Euler characteristic is additive in exact sequences (see \cite{19}, Exercise 18.4.A) and that the sheaf cohomology remains the same after doing the pushforward by an affine morphism (see \cite{4}, Exercise III.5.1), we get:

\begin{align}\label{t}
    \chi(\mathcal{O}_E) = \sum_{i=0}\chi(\mathcal{O}_{E_i}) - \sum_{i=0,j>i}\chi(\mathcal{O}_{E_i \cap E_j}) + \sum_{i=0,j>i,k>j}\chi(\mathcal{O}_{E_i \cap E_j \cap E_k})
\end{align}
Since $\overline{f}$ and $g$ are flat morphisms, we get $\chi(\mathcal{O}_{\mathfrak{X}_0}) = \chi(\mathcal{O}_{\mathfrak{X}_t})$ and $\chi(\mathcal{O}_{E}) = \chi(\mathcal{O}_{\mathfrak{W}_t})$, (see \cite{19}, Theorem 24.7.1). Besides, on one hand, it is known that the exponents are constant under $\mu$-constant deformation of $f_0$, see \cite{20} and \cite{16} and, on the other hand, the geometric genus equals the number of the exponents not greater than one, see \cite{14}. Then, $p_g(X_0) = p_g(X_t)$. 
Moreover, Mayer-Vietoris gives (see \cite{9}, proof of Theorem~1):
\begin{align*}
    \chi(\mathcal{O}_{E_0}) = \chi(\mathcal{O}_{\mathfrak{X}_0}) - p_g(X_0),\\
    \chi(\mathcal{O}_{\mathfrak{W}_t}) = \chi(\mathcal{O}_{\mathfrak{X}_t}) - p_g(X_t)
\end{align*}
Thus, we get $\chi(\mathcal{O}_{E_0}) = \chi(\mathcal{O}_{\mathfrak{W}_t}) = \chi(\mathcal{O}_{E})$.
And using, (\ref{t}), we get
\begin{align}\label{tu}
    \sum_{i=1}\chi(\mathcal{O}_{E_i}) - \sum_{i=0,j>i}\chi(\mathcal{O}_{E_i \cap E_j}) + \sum_{i=0,j>i,k>j}\chi(\mathcal{O}_{E_i \cap E_j \cap E_k}) = 0
\end{align}

Furthermore, Hodge decomposition on compact Kähler manifolds gives:

\begin{align}\label{212}
    \chi_T(E_i) = 2\chi(\mathcal{O}_{E_i}) - 2h^{1,0}(E_i) + h^{1,1}(E_i), \\
    \chi_T(E_i\cap E_j) = 2\chi(\mathcal{O}_{E_i\cap E_j})\\
    \chi_T(E_i\cap E_j \cap E_k) = \chi(\mathcal{O}_{E_i\cap E_j \cap E_k}).
\end{align}
Since we have seen in Theorem \ref{2imp} that for $i \neq 0$, $E_i$ are pure Hodge structures of type (1,1), we have, in particular, $\chi_T(E_i) = 2\chi(\mathcal{O}_{E_i}) - 2h^{1,0}(E_i) + b_2(E_i)$.

Substituting these equalities in (\ref{tu}),we get:

\begin{align*}
    \hspace*{-5.8mm} \frac{1}{2}\sum_{i > 0}\chi_T(E_i) + \sum_{i > 0}h^{1,0}(E_i) - \frac{1}{2}\sum_{i > 0} b_2(E_i) - \frac{1}{2}\sum_{\substack{i = 0\\j > i}}\chi_T(E_i \cap E_j) + \sum_{\substack{i=0\\j > i\\k>j}}\chi_T(E_i \cap E_j \cap E_k) = 0.
\end{align*}
Developing the terms of the equations, we get
\begin{align*}
    \hspace{-5.8mm} \sum_{i > 0}[b_0(E_i) - \frac{1}{2}b_1(E_i)] - \sum_{\substack{i = 0\\j > i}}[b_0(E_i \cap E_j) - \frac{1}{2}b_1(E_i \cap E_j)] + \sum_{\substack{i=0\\j > i\\k>j}}b_0(E_i \cap E_j \cap E_k) = 0.
\end{align*} 
Moreover, Theorem \ref{1imp} says that 
\begin{align*}
    \sum_{i > 0}b_0(E_i) - \sum_{\substack{i = 0\\j > i}}b_0(E_i \cap E_j) + \sum_{\substack{i=0\\j > i\\k>j}}b_0(E_i \cap E_j \cap E_k) = 0.
\end{align*} Thus, we get 
\begin{align}\label{213}
    \sum_{i > 0} b_1(E_i) = \sum_{\substack{i = 0\\j > i}} b_1(E_i \cap E_j)
\end{align}

Then, the kernel of the surjective morphism $H^1(E(1)) \to H^1(E(2))$ of the first page of Steenbrink spectral sequence has the same dimension as $H^1(E_0)$, so in particular, $H^1(\mathfrak{W}_t) \cong H^1(E_0)$. Moreover, the spectral sequence degenerating to the cohomology of $E$ gives $H^1(\mathfrak{W}_t) \cong H^1(E)$, so we conclude that $H^1(E) \cong H^1(E_0)$. 

\end{proof}

\begin{corollary}
    $K^2_{\mathfrak{W}_t} - K^2_{E_0} = b_2(E_0) - b_2(\mathfrak{W}_t)$.
\end{corollary}
\begin{proof}
    Noether formula gives the result, since $\chi(\mathcal{O}_{E_0}) = \chi(\mathcal{O}_{\mathfrak{W}_t})$ and  $b_i(\mathfrak{W}_t) = b_i(E_0)$ for $i \neq 2$.
\end{proof}

\begin{theorem}
    $\sum_{i=0}E^3_i = -3b_0(E(3))$.
\end{theorem}

\begin{proof}
    On one hand, making use of previous Corollary and the second row of the first page of Steenbrink spectral sequence, we get \begin{align} \label{em}
    K^2_{\mathfrak{W}_t} - K^2_{E_0} = -\sum_{i=1}b_2(E_i) - b_0(E(3)) + 2b_0(E(2)).
    \end{align} 
    On the other hand, using the technique of \cite{8}, we have
\begin{align*}
    K^2_{\mathfrak{W}_t} = K^2_{\mathfrak{W}} \cdot \sum_{i=0} E_i = \sum_{i=0}K^2_{\mathfrak{W}_{|E_i}} = \sum_{i=0}(K_{E_i} + \sum_{j \neq i} E_{j{|E_i}})^2 = \\
    = \sum_{i=0} K^2_{E_i} + 2\sum_{i=0}K_{E_i} \cdot (\sum_{j \neq i} E_{j{|E_i}}) +  \sum_{i=0}(\sum_{j\neq i} E_{j{|E_i}})^2 
\end{align*}
Thus, we get
    \begin{align*}
    K^2_{\mathfrak{W}_t} - K^2_{E_0} = \sum_{i=1} K^2_{E_i} + 2\sum_{i=0}K^2_{E_i} \cdot (\sum_{j \neq i} E_{j{|E_i}}) +  \sum_{i=0}(\sum_{j\neq i} E_{j{|E_i}})^2 
    \end{align*}
Now, let's develop these terms. First, using Noether formula and equalities from (\ref{212}), (\ref{213}) and Theorem \ref{1imp}, we get
    \begin{align*}
        \hspace*{-0.4cm}\sum_{i=1} K^2_{E_i} = \sum_{i=1} \left( 12 \chi(\mathcal{O}_{E_i}) - \chi_{T}(E_i) \right ) = 10b_0(E(2)) - 10b_0(E(3)) - 4b_1(E(2)) - \sum_{i=1}b_2(E_i)
    \end{align*}
Second, since $E_{j{|E_i}}$ is a curve on a smooth surface, its arithmetic genus is given by the adjunction formula hence, 
    \begin{align*}
        2\sum_{i=0}K_{E_i} \cdot (\sum_{j \neq i} E_{j{|E_i}}) = 2\sum_{\substack{i = 0\\j \neq i}} \left (2b_1(E_i \cap E_j) - 4b_0(E_i \cap E_j) - 2(E_{j{|E_i}})^2\right ) = \\
        = 4b_1(E(2)) - 8b_0(E(2)) - 2\sum_{\substack{i = 0\\j \neq i}}(E_{j{|E_i}})^2.
    \end{align*}
Third, 
\begin{align*}
    \sum_{i=0}(\sum_{j\neq i} E_{j{|E_i}})^2 = \sum_{\substack{i = 0\\j \neq i}}(E_{j{|E_i}})^2 + 6b_0(E(3)).
\end{align*}
Putting everything together, we get:
    \begin{align}\label{er}
    K^2_{\mathfrak{W}_t} - K^2_{E_0} = 2b_0(E(2)) - 4b_0(E(3)) - \sum_{i=1}b_2(E_i) - \sum_{\substack{i = 0\\j \neq i}}(E_{j{|E_i}})^2.
    \end{align}
\end{proof}
 Equaling (\ref{em}) and (\ref{er}), we get:
    \begin{align}
        3b_0(E(3))  + \sum_{\substack{i = 0\\j \neq i}}(E_{j{|E_i}})^2 = 0
    \end{align}
And since $(\sum_{i=0}E_i)^3 = 0$, we get 
    \begin{align}
        3b_0(E(3))  + \sum_{i=0}E_i^3 = 0.
    \end{align}

\end{document}